\documentclass[graybox]{svmult}

\usepackage{mathptmx}       % selects Times Roman as basic font
\usepackage{helvet}         % selects Helvetica as sans-serif font
\usepackage{courier}        % selects Courier as typewriter font
\usepackage{type1cm}        % activate if the above 3 fonts are
\usepackage{makeidx}         % allows index generation
\usepackage{graphicx}        % standard LaTeX graphics tool
\usepackage{multicol}        % used for the two-column index
\usepackage[bottom]{footmisc}% places footnotes at page bottom
\usepackage{amssymb}
\usepackage{amsmath}
\usepackage{enumerate}
\usepackage{overpic}

\usepackage{graphics}

\makeindex             % used for the subject index
\newcommand{\fcal}{\ensuremath{\mathcal{F}}}

\newcommand{\Nbb}{\ensuremath{\mathbb{N}}}
\newcommand{\Zbb}{\ensuremath{\mathbb{Z}}}

\newcommand{\Rbb}{\ensuremath{\mathbb{R}}}

\DeclareMathOperator{\divv}{div} 

\begin{document}

\title*
{Some applications of the extended Bendixson-Dulac Theorem}
\author{Armengol Gasull and Hector Giacomini}
\institute{Armengol Gasull\at Departament de Matem\`{a}tiques. Universitat Aut\`{o}noma de
Barcelona, Edifici C 08193 Bellaterra, Barcelona. Spain,
\email{gasull@mat.uab.cat} \and Hector Giacomini \at Laboratoire de Math\'{e}matiques
et Physique Th\'{e}orique.
 Facult\'{e} des Sciences et Techniques. Universit\'{e} de Tours, C.N.R.S. UMR 7350. 37200 Tours.
France, \email{Hector.Giacomini@lmpt.univ-tours.fr}}
\maketitle

\abstract*{During the last years the authors have studied the number of limit
cycles of several families of planar vector fields. The common tool has been the
use of an extended version of the celebrated  Bendixson-Dulac Theorem.
 The aim of this work is to present an unified approach of some of these results,
 together with their corresponding proofs. We also provide several applications.}

\abstract{During the last years the authors have studied the number of limit
cycles of several families of planar vector fields. The common tool has been the
use of an extended version of the celebrated  Bendixson-Dulac Theorem.
 The aim of this work is to present an unified approach of some of these results,
 together with their corresponding proofs. We also provide several applications.}

\section{The Bendixson-Dulac Theorem}
\label{sec:1}

 Ivar Bendixson  and Henri Dulac are the fathers of the today known as Bendixson-Dulac
 Theorem. The classical version of this theorem appears in most textbooks on differential
 equations; see~\cite{DLA,Ye,Z2} with many applications.  Let us recall
it.
 Consider  a $\mathcal{C}^1$-planar differential system
 \begin{equation}\label{ee1}
\dot x=P(x,y),\quad \dot y=Q(x,y),
 \end{equation}
defined in some open simply connected subset $\mathcal{U}\subset\mathbb{R}^2$, and
set $X=(P,Q).$ Assume that  there exists a $\mathcal{C}^1$ function
$D:\mathcal{U}\to\mathbb{R}$, such that
\[
\left.\operatorname{div}\left(D\, X\right)\right|_{\mathcal{U}}=
\left.\dfrac{\partial (D(x,y)P(x,y))}{\partial x}+\dfrac{\partial
(D(x,y)Q(x,y))}{\partial y}\right|_{\mathcal{U}}\ge0\quad (\mbox{or}\quad \le0),
\]
 vanishing only on a set of zero Lebesgue measure
 . Then~ system \eqref{ee1} has no periodic orbits  contained in $\mathcal{U}$. This
function $D$ is usually called a {\it Dulac function} of the system.

This theorem has been extended to multiple connected regions, see for
instance~\cite{Che2,GGi,Lloyd,Yamato} obtaining  then a method for determining
upper bounds of the number of limit cycles in $\mathcal{U}$. In the next section we
recall this extension and present the proof given in~\cite{GGi}.

As we will see  this extension  can be used  if it is possible to find a suitable
function~$V$, a real number $s,$ and a  domain $\mathcal{U}\subset\mathbb{R}^2$
such that
\[ M=\left. \frac{\partial V}{\partial x}P+\frac{\partial V}{\partial y}Q
  +s\Big( \frac{\partial P}{\partial x}+ \frac{\partial Q}{\partial y}
  \Big)V\right|_{\mathcal{U}}\]
  does not change sign and vanishes  on a set of zero Lesbesgue measure.
  Moreover, the upper bound given by the method
  for the  number of limit cycles depends on the number and distribution of the ovals of
  $\{V(x,y)=0\}$ in~$\mathcal{U}.$

When all the involved functions $P,Q$ and $V$ are polynomials this approach relates
both parts of Hilbert's Sixteenth Problem. Recall that the first part deals with
the number and distribution of ovals of a real algebraic curve  in terms of its
degree while the second part asks to find an uniform bound of the number of limit
cycles of systems of the form~\eqref{ee1} when both polynomials have a given
degree; see~\cite{I,Wilson}.

Notice that the importance of the use of the Bedixson-Dulac results is that in many
cases they translate the problem of knowing the number of periodic solutions of a
planar polynomial differential equation to a problem of semi-algebraic nature: the
control of the sign of a  polynomial in a suitable domain.

Analogously to Lyapunov functions,  the first difficulty to apply these results is
 to find a suitable Dulac function. The problem of its existence,  in the basin of
attraction of critical points, is treated in \cite{cha}. A second  difficulty of
the method is to find a suitable region $\mathcal U$.

The aim of this paper is to present an unified point of view of some of the results
obtained by the authors in~\cite{GGi,GGi2,GGi3,GGiL}, together with some proofs.
These results give methods to find Dulac functions $D$, or equivalently functions
$V$ and values $s$, for which the corresponding expression $M$ is simple and so its
sign can be easily studied. We also apply the method to give an upper bound of the
number of limit cycles for several families of planar systems.

The Bendixson-Dulac approach has been extended in several directions: to prove
non-existence of periodic orbits in higher dimensions, see~\cite{F,LM}; to control
the number  of isolated  periodic solutions of some non-autonomous Abel
differential equations, see for instance~\cite{a,Che7}; to prove non-existence of
periodic orbits for some difference equations, see~\cite{MM}.

We can not end this introduction without talking about the contributions on the use
of Dulac functions    of our friend and colleague Leonid Cherkas, sadly recently
deceased. His important work in this subject started many years ago and arrives
until the actuality, continued by his collaborators, see for
instance~\cite{Che1,Che2,Che3,Che4,Che5,Che6,Che7} and the reference therein. In
fact, one of the main motivations for the fist author  to work in this direction
were the pleasant  conversations with him walking around the beautiful gardens of
the Beijing University in the summer of 1990.

\subsection{The Bendixson-Dulac Theorem for multiple connected regions}

 An open subset $\mathcal U$ of $ \mathbb{R}^2$  with smooth boundary, is said to be
$\ell${\it-connected}\, if its fundamental  group, $\pi_1(\mathcal{U})$ is $\Zbb*
\stackrel{(\ell)}{\cdots}*\Zbb,$ or in other words if $\mathcal U$ has $\ell$
holes. We will say that $\ell(\mathcal{U})=\ell$. We state and prove,
following~\cite{GGi}, the extension of the Bendixson-Dulac Theorem to more general
domains; see other proofs in~\cite{Che2,Lloyd,Yamato}. As usual,
$\langle\cdot,\cdot\rangle$ denotes the scalar product in $\mathbb{R}^2.$

\smallskip

\noindent{\bf{Extended Bendixson-Dulac Theorem.}} {\it Let $\mathcal{U}$ be an
$\ell$-connected open
 subset of $ \mathbb{R}^2$ with smooth boundary. Let $D\colon
\mathcal{U}\rightarrow \mathbb{R}$  be a $\mathcal{C}^1$ function such
that\begin{equation}\label{ebd}
    M:={\operatorname{div}}  (DX)=
\frac{\partial D}{\partial x}P+\frac{\partial D}{\partial y}Q +D\,( \frac{\partial
P}{\partial x}+ \frac{\partial Q}{\partial y} )= \langle \nabla D,X \rangle
+D\,{\operatorname{div}}  (X) \end{equation} does not change sign in $\mathcal U$
and vanishes only on a null measure Lebesgue set,  such that $\{M=0\} \cap\{D=0\}$
does not contain periodic orbits of (\ref{ee1}). Then the maximum number of
periodic orbits of (\ref{ee1}) contained in $\mathcal{U}$ is $\ell.$ Furthermore
each one of them is a hyperbolic limit cycle that does not cut
 $ \{D=0\}$ and its
stability  is given by the sign of $DM$ over it.}

\begin{proof}
Observe that $M|_{\{D=0\}}=\langle \nabla D,X \rangle|_{\{D=0\}} \ge0$ does not
change sign in $\mathcal{U}$. Since, by hypothesis, there are no periodic orbits
of (\ref{ee1}) contained in $\{M=0\} \cap\{D=0\}$, we have that the periodic
orbits of (\ref{ee1}) do not cut $\{D=0\}.$

If $\mathcal U$ is simply connected ($\ell=0$) then by the  Bendixson-Dulac Theorem
we have that (\ref {ee1}) has no periodic orbits in $\mathcal U$. We give now a
proof for an arbitrary $\ell.$ Assume that system (\ref {ee1}) has $\ell+1$
different periodic orbits $\gamma_i$, included in $\mathcal U$. These orbits induce
$\ell+1$ elements $\overline\gamma_i$ in the first homology group of $\mathcal{U},
H_1(\mathcal{U})=\Zbb \oplus\stackrel{(\ell)}{\cdots}\oplus\Zbb$. Since this group
has at most $\ell$ linearly independent elements it follows that there is a non
trivial linear combination of them giving $0\in H_1(\mathcal{U})$. Then
$\sum_{i=1}^{\ell+1}m_i\overline\gamma_i=0$, with $(m_1,\ldots,m_{\ell+1})\ne0$.

This last fact means that the curve $\sum_{i=1}^{\ell+1}m_i\gamma_i$ is the
boundary of a two cell $C$ for which Stokes Theorem can be applied. Then
$$\iint_C {\operatorname{div}}
(DX)=\int_{\sum_{i=1}^{\ell+1}m_i\gamma_i}\langle D\,X,\bold n\rangle.$$ Note that
the right hand term in this equality is zero because $DX$ is tangent to the curves
$\gamma_i$ and that the left one is non-zero by our hypothesis. This fact leads to
a contradiction. So $\ell$ is the maximum number of periodic orbits of (\ref {ee1})
in $\mathcal U$.

Let us prove their hyperbolicity. Fix one periodic orbit $\gamma=\{(x(t),y(t)),
t\in[0,T]\}\subset \mathcal{U},$ where $T$ is its period. Remember that
$\gamma\cap\{D=0\}=\emptyset.$ In order to study its hyperbolicity and stability
we have to compute $\int_0^T{\operatorname{div}}    X(x(t),$ $y(t))dt,$ and to
prove that it is not zero. This fact follows by integrating the equality $$
{\operatorname{div}} X= \frac{\partial P}{\partial x}+ \frac{\partial Q}{\partial
y}=\frac{{\operatorname{div}} (D\,X)}D-\frac{
 \frac{\partial D}{\partial x}P+\frac{\partial D}{\partial y}Q}D, $$
because the last term of the  right hand side of the above equality coincides with
$\frac{d}{dt}\ln| D(x(t),y(t))|.$\qed
\end{proof}

To apply the above theorem,  we consider  a function $D(x,y)$ of the form
$|V(x,y)|^{m}$ where $V$ is a smooth function in two variables in $ \mathbb{R}^2$
and $m$ is a real number.

Before giving the result for this particular choice of $V$ we introduce some more
notation. Given an open subset $ \mathcal{W}$ with smooth boundary and a smooth
function $V:{\mathcal{W}} \to \mathbb{R}$ we denote by $\ell( \mathcal{W},V)$ the
sum of $\ell(\mathcal{U})$ where $\mathcal U$ ranges over all the connected
components of ${\mathcal{W}}\setminus\{V=0\}.$ Finally, we denote by
$c({\mathcal{W}},V)$ the number of closed ovals of $\{V=0\}$ contained in
${\mathcal{W}}.$ See Figure~\ref{fig:1} for an illustration of these definitions.

\begin{figure}[h]
%\begin{center}
\begin{picture}(0,0)%
\includegraphics[scale=0.8]{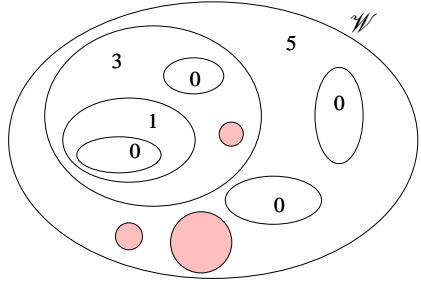}%
\end{picture}%
\setlength{\unitlength}{3315sp}%{4144sp}%
\begingroup\makeatletter\ifx\SetFigFont\undefined%
\gdef\SetFigFont#1#2#3#4#5{%
  \reset@font\fontsize{#1}{#2pt}%
  \fontfamily{#3}\fontseries{#4}\fontshape{#5}%
  \selectfont}%
\fi\endgroup%
\begin{picture}(3076,2084)(488,-1508)
\put(3151,344){\makebox(0,0)[b]{\smash{{\SetFigFont{10}{12.0}{\rmdefault}{\mddefault}{\updefault}{\color[rgb]{0,0,0}$\mathcal{W}$}%
}}}}
\put(1576,-376){\makebox(0,0)[b]{\smash{{\SetFigFont{8}{12.0}{\rmdefault}{\mddefault}{\updefault}{\color[rgb]{0,0,0}1}%
}}}}
\put(2971,-241){\makebox(0,0)[b]{\smash{{\SetFigFont{8}{12.0}{\rmdefault}{\mddefault}{\updefault}{\color[rgb]{0,0,0}0}%
}}}}
\put(2521,-1006){\makebox(0,0)[b]{\smash{{\SetFigFont{8}{12.0}{\rmdefault}{\mddefault}{\updefault}{\color[rgb]{0,0,0}0}%
}}}}
\put(1441,-601){\makebox(0,0)[b]{\smash{{\SetFigFont{8}{12.0}{\rmdefault}{\mddefault}{\updefault}{\color[rgb]{0,0,0}0}%
}}}}
\put(1891,-61){\makebox(0,0)[b]{\smash{{\SetFigFont{8}{12.0}{\rmdefault}{\mddefault}{\updefault}{\color[rgb]{0,0,0}0}%
}}}}
\put(1306, 74){\makebox(0,0)[b]{\smash{{\SetFigFont{8}{12.0}{\rmdefault}{\mddefault}{\updefault}{\color[rgb]{0,0,0}3}%
}}}}
\put(2611,209){\makebox(0,0)[b]{\smash{{\SetFigFont{8}{12.0}{\rmdefault}{\mddefault}{\updefault}{\color[rgb]{0,0,0}5}%
}}}}
\end{picture}%
%\end{center}
\hspace{15mm} \sidecaption \caption{Open set $\mathcal{W}$ with
$\ell(\mathcal{W})=3$. The grey cercles are holes in $\mathcal{W}$  and the thick
lines correspond to $\{V=0\}.$ The numbers
 displayed are the values $\ell(\mathcal{U})$ for each connected component $\mathcal U$ of
$\mathcal{W}\setminus\{V=0\}.$ For this example $c(\mathcal{W},V)=6$ and
$\ell(\mathcal{W},V)=9.$}
\label{fig:1}       % Give a unique label
\end{figure}

\begin{corollary}\label{B}
Assume that there exist a real number $s$ and an analytic function $V$ in $
\mathbb{R}^2$ such that
\begin{equation*}
    M_s:=  \frac{\partial V}{\partial x}P+\frac{\partial V}{\partial y}Q
  +s\Big( \frac{\partial P}{\partial x}+ \frac{\partial Q}{\partial y} \Big)V=
  \langle \nabla V,X \rangle +sV{\operatorname{div}}   (X)
\end{equation*} does not change sign in an open region $\mathcal{W}\subset \Rbb^2$
with regular boundary and vanishes only in a null measure Lebesgue set. Then the
limit cycles of system (\ref{ee1}) are either totally contained in
$\mathcal{V}_0:=\{V=0\},$ or do not intersect  $\mathcal{V}_0$.

Moreover the number of limit cycles contained in  $\mathcal{V}_0$ is at most
$c(\mathcal{W},V)$ and the number $N$ of limit cycles that do not intersect
$\mathcal{V}_0$ satisfies
\[
N\le \begin{cases} \ell(\mathcal{W})\quad &\mbox{if}\quad s>0,\\
0&\mbox{if}\quad s=0,\\
\ell(\mathcal{W},V)\quad &\mbox{if}\quad s<0,\\
\end{cases}
\] Furthermore for any $s\ne0$ the limit cycles of this second type are
hyperbolic.

\end{corollary}

\begin{proof} First observe that since  $M_s$ does not change sign we have
that  on the analytic curves $\mathcal{V}_0$,  $ \langle \nabla V,X \rangle $ does
not change sign. Therefore these curves are either solutions of (\ref{ee1}) or
curves crossed by the flow generated by (\ref{ee1}) in just one direction. Hence
all limit cycles in $\mathcal{W}$ are either contained in the connected components
of $ \mathcal{W}\setminus \mathcal{V}_0$ or in $\mathcal{V}_0$. This fact implies
the first assertions of the Theorem. In order to bound the number of limit cycles
of (\ref{ee1}) we apply the extended Bendixson-Dulac Theorem  to each one of the
connected components $\mathcal U$ of $ \mathcal{W}\setminus \mathcal{V}_0.$ The
fact that when $D=|V|^{m},$
$$
     {\operatorname{div}}   (D\,X)=\langle \nabla D,X \rangle +D{\operatorname{div}}
       (X)=
    \mbox{sign}(V)m|V|^{m-1}\left[\langle \nabla V,X \rangle +\frac1m
     V{\operatorname{div}}    (X)\right],
$$ gives  the theorem  by taking $m=1/s.$ Observe that the difference between the
cases $s>0$ and $s<0$ comes from the fact that in the first case the function $D$
is well defined in the whole plane. For the case $s=0$ the proof is easier because
$M_0=dV/dt=\langle \nabla V,X \rangle.$\qed
\end{proof}

The above corollary shows that the  study of the functions
\begin{svgraybox}
\begin{equation}\label{ms}
    M_s=  \frac{\partial V}{\partial x}P+\frac{\partial V}{\partial y}Q
  +s\Big( \frac{\partial P}{\partial x}+ \frac{\partial Q}{\partial y} \Big)V
\end{equation}
\end{svgraybox}
\noindent gives a  tool for controlling the number of limit cycles of
system~\eqref{ee1}. As we will see this approach turns out to be useful for  many
families of planar vector fields. This function is  also often used in the quoted
works of Cherkas and his coauthors.

\subsection{Some simple examples}

As  paradigmatic examples we will give short and easy proofs  of the non-existence
of limit cycles for a generalization of the Lotka-Volterra system  and   of the
uniqueness of the limit cycle of the van der Pol system. The first  one is a
folklore prove and the second one is given by Cherkas, see \cite[p. 105]{Chi}. We
will prove  also a more general non-existence result for Kolmogorov systems.

\subsubsection{Non-existence  of  limit cycles for some predator-prey systems}

Consider the following extension of the celebrated Lotka-Volterra system
\begin{equation}\label{lv}
\dot x=x(ax+by+c),\quad \dot y=y(dx+ey+f),
\end{equation}
where all the parameters are real numbers. It appears in most texts books of
mathematical ecology. By uniqueness of solutions it is clear that if it  has
periodic orbits then they do not intersect the coordinate axes. By making the
change of variables $x\to\pm x,$ $y\to\pm y,$ if necessary, we can restrict our
attention to the first quadrant $\mathcal U$ and prove that the system has no
periodic orbit in it. To do this consider the Dulac function $D(x,y)=x^Ay^B,$ where
the real numbers $A$ and $B$ have to be determined. Then the function $M$ appearing
in~\eqref{ebd} is
\begin{align*}
M(x,y)&=\langle \nabla D(x,y),X(x,y) \rangle +D(x,y)\,\operatorname{div}X(x,y) \\
&= x^A y^B\left((aA+dB+2a+d)x+(bA+eB+2e+b)y+(cA+fB+c+f)  \right).
\end{align*}
When  $ae-bd\ne0$ we can solve the linear system obtained vanishing the
coefficients of $x$ and $y$ with unknowns $A$ and $B$. Call the solution $A=\alpha$
and $B=\beta$. Then
\[
M(x,y)=\frac{abf+ced-aef-ace}{ae-bd}x^{\alpha}y^{\beta}:=R x^{\alpha}y^{\beta}.
\]
When $R\ne0$ we can apply the Bendixson-Dulac Theorem and since $\mathcal U$ is
simply connected ($\ell(\mathcal{U})=0$) the system has no limit cycles. When $R=0$
then it $x^{\alpha}y^{\beta}$ is an integrating factor. Hence the system is
integrable and its first integral  is smooth in $\mathcal U$. Thus it can not have
isolated periodic orbits, i.e. it has no limit cycles. This case includes  the
famous Lotka-Volterra system. Recall that it  has a center in $\mathcal U$,
surrounded by periodic orbits.

When  $ae-bd=0$ then either the linear system $ax+by+c=0,$ $dx+ey+f=0,$ with
unknowns $x$ and $y$ has no solutions or its solutions are  either the full plane
 or a whole  line. In the first case the only critical points of system~\eqref{lv}
are on the axes, so the system can not have  periodic orbits. Otherwise it is
either the trivial system $\dot x=0,\, \dot y=0$ or is a reparameterization of the
simple system $\dot x=g x,\, \dot y= h y,$ for some real numbers $g,h$, which
clearly can not have periodic orbits either.

\subsubsection{Non-existence  of  limit cycles for a class of Kolmogorov systems}

Following~\cite{a2} we  give a  non-existence criterion  for a family of
Kolmogorov systems. This result can be applied to the  Gause-type systems
considered in \cite{Moreira} or to the systems studied in~\cite{siam}.

\begin{proposition}\label{Thm-Mas2}
Consider the $\mathcal{C}^1$-system
\begin{equation}\label{ec-masgeneral2}
\dot x=x\left(g_0(x)+g_1(x)y\right),\quad \dot
y=y\left(h_0(x)+h_1(x)y+h_2(x)y^2\right),
\end{equation}
for  $x\ge 0$, $y\ge 0$. For any $\lambda\in\mathbb {R}$ define the functions:
\begin{align*}
S_\lambda(x)&=x[g'_0(x)g_1(x)-g_0(x)g'_1(x)]+\lambda
h_0(x)g_1(x)-(1+\lambda)g_0(x)h_1(x),\\
T_\lambda(x)&=(2+\lambda) h_2(x)g_1(x).
\end{align*}
Let $\mathcal{I}\subset\mathbb{R}^+$ be an open interval. Assume that there exists
a value of $\lambda$ such that $S_\lambda(x)T_\lambda(x)\ge0$, for all $x\in
\mathcal{I}$, and all its zeroes  are isolated. Then system \eqref{ec-masgeneral2}
does not have periodic orbits  in the strip
$\mathcal{U}=\mathcal{I}\times(0,+\infty)$.
\end{proposition}

\begin{proof} First, let us prove that  if the system has a limit cycle then it
can not intersect the set $\{(x,y)\,|\,x>0,g_1(x)=0\}.$ This holds because if
$\bar x>0$ is such that $g_1(\bar x)=0$ then either $x=\bar x$ is an invariant
line (i.e. also $g_0(\bar x)=0$) or it is a line without contact, i.e $\dot
x\,|_{x=\bar x}= \bar xg_0(\bar x)\ne0$. Hence, in the region
where~\eqref{ec-masgeneral2} can have periodic orbits we can always assume that
$g_1$ does not vanish.

Consider now the family of Dulac functions  $D(x,y)=y^{\lambda-1} Z(x),$ where
$\lambda$ is given in the statement and $Z$ is an unknown function. Computing the
function~\eqref{ebd},
\begin{align*}
M(x,y)&=\, \operatorname{div}\big(D(x,y)X(x,y)\big)\\
&=\big[\left(xg_0(x)Z(x)\right)'+\lambda h_0(x)Z(x)
+\left(\left(xg_1(x)Z(x)\right)'+(\lambda+1) h_1(x)Z(x)\right)y\\
&\qquad+\big((\lambda+2) h_2(x)Z(x)\big)y^2\big]
y^{\lambda-1}.\\
\end{align*}
The solutions of the differential equation
\begin{equation}\label{edo} \left(xg_1(x)Z(x)\right)'+(\lambda+1)
h_1(x)Z(x)=0
\end{equation}
are
\[
Z_{x_0}(x)=\frac{\exp\left[-(\lambda+1){\displaystyle\int_{x_0}^x}
\dfrac{h_1(s)}{sg_1(s)}\,ds \right]}{x g_1(x)},
\]
where $x_0>0$ is an arbitrary contant. By taking the Dulac function $\widetilde
D(x,y)=y^{\lambda-1}  Z_{x_0}(x),$ for a given $x_0>0$, and  taking into account
that $Z_{x_0}(x)$ satisfies~\eqref{edo}, we obtain after some computations that
\[
M(x,y)=\frac{Z_{x_0}(x)}{g_1(x)}\left(S_\lambda(x)+T_\lambda(x) y^2
\right)y^{\lambda-1}.
\]

Since on $\mathcal I$, $S_\lambda(x)T_\lambda(x)\ge0$ we have proved that $M$ does
not change sign in $\mathcal{U}=I\times(0,+\infty)$, which is simply connected, and
vanishes on a set of zero Lebesgue measure given by some vertical straight lines.
Hence, by the  Bendixson-Dulac Theorem, the result is proved. \qed\end{proof}

Observe that the function $S_\lambda(x)$ of Proposition~\ref{Thm-Mas2} can also be
written as
\[
S_\lambda(x)=g_1^2(x)\left[ x\left(\frac{g_0(x)}{g_1(x)}\right)'+\lambda
\frac{h_0(x)}{g_1(x)}-(1+\lambda)\frac{g_0(x)h_1(x)}{g_1^2(x)}\right].
\]
When $h_1(x)\equiv0$, it essentially coincides with the one given in the
non-existence criterion presented in \cite[Thm 4.1]{K}.

\subsubsection{Uniqueness of the limit cycle for the van der Pol
equation}\label{svdp}

 The second order van der Pol equation $ \ddot x
+\varepsilon(x^2-1)\dot x+ x = 0, $ can be written as the planar system
\begin{equation*}
\dot x = y, \qquad \dot y= -\varepsilon(x^2-1)y-x.
\end{equation*}
Taking $V(x,y)=x^2+y^2-1$ we obtain that the associated function $M_s$ given
in~\eqref{ms} is
\[
M_s(x,y)=-\varepsilon(x^2-1)(sx^2+(2+s)y^2-s).
\]
Choosing $s=-2$ we get that $M_{-2}(x,y)=2\varepsilon(x^2-1)^2$. So, for
$\varepsilon\ne0$,  this function does not change sign and vanishes on two straight
lines. Hence since, $\ell(\mathbb{R}^2,V)=1$ and $s<0$, by Corollary~\ref{B}, we
obtain the van der Pol system has at most one limit cycle, which when exists is
hyperbolic, and lies outside the unit circle. This approach does not provide the
existence of the limit cycle. The existence, for $\varepsilon\ne0$, can be obtained
studying the behavior of the flow at infinity.

% In most cases we will use special algebraic functions of the form
%$$f(x,y)=g_0(x)+g_1(x)y+g_2(x)y^2+\cdots +g_n(x)y^n ,$$ with $g_i(x)$ polynomials.
%For these functions the number $ \ell(\Rbb^2,f)$ is sometimes not difficult to
%compute, and therefore Proposition \ref{B} is easy to apply.

\section{Control of  the function $M_s$}
%\label{sec:3}

To apply the Dulac method to concrete examples the main difficulty is to find a
suitable couple $s\in\mathbb{R}$ and $V$ and then control the sign of the function
$M_s$ given in~(\ref{ms}). Many times a good trick consists in trying to reduce the
question to a one variable problem. This approach is developed  in
Subsection~\ref{sub1} following \cite{GGi,GGi2}.

Another point of view is to work in polar coordinates. Then the control of the
corresponding function $M_s$ takes advantage of writing the functions as
polynomials of the radial component with coefficients depending periodically on
the angle. This approach has been followed in~\cite{GGi3} and some results are
presented in Subsection~\ref{sub2}.

\subsection{The function $M_s$ is reduced to a one variable function}\label{sub1}

\subsubsection{A first method}\label{ss1}

\begin{proposition}\label{MT}
Consider a $\mathcal{C}^1$ system of the form
\begin{equation}
\label{me}
  \dot x= p_0(x)+p_1(x)y=P(x,y),\quad
  \dot y= q_0(x)+q_1(x)y + q_2(x) y^2=Q(x,y),
\end{equation}
with $p_1(x)\not\equiv0.$ For each $s\in\Rbb$ and for each $n\in\Nbb$   it is
possible to associate to it  a $(n+1)$-parameter family of functions
$V_n(x,y;c_0,c_1,\ldots,c_{n}):=V_n(x,y)$ of the form
$$
V_n(x,y)=v_0(x)+v_1(x)y+v_2(x)y^2+\cdots +v_n(x)y^n,
$$
such that for each one of them the function~\eqref{ms},
\[
M_{s,n}(x)=\langle \nabla V_n,(P,Q) \rangle +s
     V_n \operatorname{div} (P,Q)\,
\]
is a function only of the $x$-variable.
\end{proposition}

\begin{proof}
Direct computations give
\begin{align*}
& \langle \nabla V_n,(P,Q) \rangle +s
     V_n \operatorname{div} (P,Q)=\\ &\hspace*{1cm}=\left[\left\{\left(sp'_1+2sq_2+nq_2\right)v_n +
p_1v'_n\right\}
   y^{n+1}+\fcal_n(v_n,v_{n-1})y^n+\right.\\
&\hspace{1cm}+\left.\left\{\fcal_{n-1}(v_{n-1},v_{n-2})+
nq_0(x)v_n(x)\right\}y^{n-1} +\cdots +\right.\\ &\hspace{1cm}+\left.
\left\{\fcal_1(v_1,v_0)+2q_0(x)v_2(x)\right\}y +
    \left\{(sp'_0+sq_1)v_0 + p_0v'_0 + q_0v_1\right\}\right],
\end{align*}
where for each $j=1,2\ldots,n,$
\begin{align*}
    \fcal_j(v_j,v_{j-1},v'_j,v'_{j-1})=&
    (sp'_{0}+sq_1+jq_1)v_j(x)+\\&  p_0v'_j(x)+(sp'_1+2sq_2+(j-1)q_2))v_{j-1}(x)+
    p_1v'_{j-1}(x).
\end{align*}

From the above expressions we can obtain a $1$-parameter family of functions
$v_n^*(x;c_{n}):=v_n^*(x)$ such that the coefficient of $y^{n+1}$ vanishes, by
solving a linear first order ordinary differential equation. Once we have $v_n^*$,
from $\fcal_n(v_n^*,v_{n-1})=0$ we get $v_{n-1}^*(x;c_n,c_{n-1}):=v_{n-1}^*(x)$
and so on until we have found $v_n^*, v_{n-1}^*,\ldots, v_0^*$. Finally, we obtain
\[
   \langle \nabla V_n,(P,Q) \rangle +sV_n  \operatorname{div}(P,Q)=
    \left[(sp'_0+sq_1)v_0^* + p_0(v_0^{*})'
      + q_0v_1^*\right]= M_{s,n}(x),
\]
as we wanted to prove.\qed
\end{proof}

\begin{corollary}\label{corolari}
Consider the generalized Li\'enard system
\[
\dot x=y-F(x):=P(x,y),\quad\, \dot y= -g(x):=Q(x,y).\]
 If we take
\[
V_2(x,y)=\left(\frac{s(s+1)}2(F(x))^2+c_1sF(x)+2G(x)+c_0\right)+(s F(x)+c_1)y+y^2,
\]
where $G(x)=\int_0^xg(z)\,dz,$ then
\begin{align*}
M_{s,2}(x)&=\langle \nabla V_2,(P,Q) \rangle +s
     V_2\divv
     (P,Q)\\&=-\frac{s(s+1)(s+2)}2(F(x))^2F'(x)-s(s+1)c_1F(x)F'(x)\\
&-(s+2)g(x)F(x)-2sF'(x)G(x)-sc_0F'(x)-c_1g(x).
\end{align*}
In particular, for $s=-1$ we have
\[
V_2(x,y)=\left(-c_1F(x)+2G(x)+c_0\right)+( -F(x)+c_1)y+y^2,\] and
\[M_{-1,2}(x)= 2F'(x)G(x)+c_0F'(x)-g(x)F(x)-c_1g(x).
\]
\end{corollary}

As an application of the above results   we prove here the uniqueness and
hyperbolicity of the limit cycle of a Li\'enard system with a rational $F.$ The
uniqueness (without proving the hyperbolicity) for this system was already proved
in \cite{Conti}; see also \cite{GG}. Other applications are given in~\cite{GGi}.

\begin{proposition}\label{rational}
The Li\'enard system
\begin{equation}\label{eq301}\dot x=y-F(x),\,\quad\dot
y= -x, \quad \mbox{with}\quad   F(x)=\frac{x(1-cx^2)}{(1+cx^2)}\end{equation} and
$c$ a real positive constant, has at most one limit cycle. Furthermore, when it
exists it is hyperbolic and unstable.
\end{proposition}

\begin{proof}

We apply  Proposition~\ref{MT} and Corollary~\ref{corolari} with $s=-1,$ $n=2$ and
$V(x,y)$ given by the rational function:
\begin{equation*}\label{eq303}
  V(x,y) = y^2-F(x)y+x^2\,.
\end{equation*}
Then
\begin{equation*}
 M_{-1,2}(x)= \frac{-4c x^4}{(1+cx^2)^2}<0 \quad \mbox{for all } \quad x\not = 0\,.
\end{equation*}
The function $V(x,y)=0$ is a   second degree polynomial in the variable $y$, with
discriminant
\begin{equation*}
  \Delta= x^2\left(\left(\frac{1-cx^2}{1+cx^2}\right)^2-4\right)=
  -\frac{x^2(cx^2+3)(3cx^2+1)}{(1+cx^2)^2}<0 \qquad
  \mbox{for\quad all } x\not =0\,.
\end{equation*}
Hence  the set $\{V=0\}$  reduces to the origin. Therefore  $c(\mathbb{R}^2,V)=0$
and $\ell(\mathbb{R}^2,V)=1.$ From Corollary~{B} we conclude that system
\eqref{eq301} has at most one limit cycle. The origin is the only critical point
of this system and it is stable. Then, when the limit cycle exists it is
hyperbolic and unstable.\qed
\end{proof}

\subsubsection{A second method}

\begin{proposition}\label{lemabasic}
Consider a $\mathcal{C}^1$ system of the form
\begin{equation*}\label{general}
  \dot x= y=P(x,y),\quad
  \dot y= h_0(x)+h_1(x)y + h_2(x) y^2+y^3=Q(x,y),
\end{equation*}
and fix a positive integer number $n.$ There is a constructive procedure to obtain
an $(n+1)$-th order linear differential equation
\begin{equation}\label{lineal}
y^{(n+1)}(x)+ r_{n,n}(x)\,y^{(n)}(x)+\cdots+r_{n,1}(x)\,y'(x)+ r_{n,0}(x)\,y(x)=0,
\end{equation}
such that if  $y(x)=v_n(x)$ is any of its solutions, we can define a function
\begin{equation*}
V_n(x,y):=v_{n,0}(x)+v_{n,1}(x)y+v_{n,2}(x)y^2+\cdots +v_{n,n}(x)y^n,
\end{equation*}
where $v_{n,n}(x)=v_n(x)$ and $v_{n,i}(x), i=0\ldots n-1,$ are obtained from given
expressions involving $h_i(x), i=0,1,2, v_n(x)$    and their derivatives, such that
the corresponding function $M_s$  given in~\eqref{ms} with $s=-n/3$,
\[
M^{[n]}:=M_{-n/3}=\langle \nabla V_n,(P,Q) \rangle -\frac{n}{3} V_n
     \divv (P,Q),
\]
is a function only of the $x$-variable.
\end{proposition}

\begin{proof} For sake
of simplicity we present the details of the proof only for the case $n=2.$  Also,
for sake of brevity and during this proof, when it appears a function of the $x$
variable that we do not want to specify we simply will write $*\,$.

Take
$V_2(x,y)=v_{2,0}(x)+v_{2,1}(x)y+v_{2,2}(x)y^2:=v_{0}(x)+v_{1}(x)y+v_{2}(x)y^2.$
Then
\begin{align*}
M^{[2]}=&\langle \nabla V_n,(P,Q) \rangle -\frac{2}{3} \divv (P,Q)V_2=\\
&\left( v'_2 \left( x \right) + \frac 2 3\,{v_2} \left( x \right) {h_2} \left( x
\right) -{v_1} \left( x
 \right)  \right) {y}^{3}+\\
& \left(v'_1\left( x \right)+  \frac4 3\,{v_2} \left( x \right) {h_1} \left( x
\right) -\frac1 3\,{v_1} \left( x \right) {h_2} \left( x
\right) -2\,{v_0} \left( x \right)  \right) {y}^{2}+\\
& \left( v'_0 \left( x \right)+\frac1 3\,{v_1} \left( x
 \right) {h_1} \left( x \right) -\frac4 3\,{h_2} \left( x \right) {
v_0} \left( x \right)  +2 \,{v_2}
\left( x \right) {h_0} \left( x \right)  \right) y+\\
& \left({v_1} \left( x \right) {h_0} \left( x \right)-\frac2 3\, {h_1} \left( x
\right) {v_0} \left( x \right)\right).
\\
\end{align*}

By choosing the following expressions for $v_0$ and $v_1$
\begin{align*}\label{gg}
v_0(x)&= \frac12\left(v'_1\left( x \right)+  \frac4 3\,{\it v_2} \left( x \right)
{\it h_1} \left( x \right) -\frac1 3\,{\it v_1}
\left( x \right) {\it h_2} \left( x \right)\right),\nonumber\\
v_1(x)&= v'_2 \left( x \right) + \frac 2 3\,{\it v_2} \left( x \right) {\it h_2}
\left( x \right),
\end{align*}
we get that the coefficients of $y^2$ and $y^3$ in $M^{[2]}$ vanish. Observe that
$v_1(x)=v_2'(x)+*\,v_2(x)$ and that $v_0(x)=v_2''(x)/2+*\,v_2'(x)+*\,v_2(x).$ Hence
if we substitute these equalities in the coefficient of $y$ in the expression of
$M^{[2]}$ we get that it writes as $ v_2'''(x)/2+*\,v_2''(x)+*\,v_2'(x)+*\,v_2(x).$
By imposing that this last expression be identically zero we get the linear
ordinary differential equation~(\ref{lineal}) given in the statement of the lemma.
Hence for these values of the functions $v_i, i=0,1,2$ the expression of $M^{[2]}$
is the function of one variable
\begin{equation*}\label{MM}
M^{[2]}(x)={ v_1} \left( x \right) { h_0} \left( x \right)-\frac2 3\, { h_1} \left(
x \right) { v_0} \left( x \right),
\end{equation*}
as we wanted to prove.\qed
\end{proof}

The advantage of the above result is that for each it $n$ gives the freedom to
choose any solution of a linear ordinary differential equation of order $n+1$. Then
using it we have to prove that the corresponding $M^{[n]}$ does not change sign.
This approach is used in~\cite{GGi3} to study the particular case $a=e=0$ of the
challenging question proposed in~\cite{CGM}:

\begin{svgraybox}
{\bf Question.} Consider the planar semi-homogeneous system
\[\dot x= ax+by,\quad
  \dot y= cx^3+dx^2y+exy^2+fy^3.\]
\noindent Is two its maximum number of limit cycles?
\end{svgraybox}

\subsection{Computations in polar coordinates}\label{sub2}

To work in polar coordinates we will need the expression of $M_s$ in terms of the
expression of the vector field~\eqref{ee1} in polar coordinates,
\begin{align}\label{eq3}
  \dot r&=R(r,\theta):=P(r\cos\theta,r\sin\theta)\cos\theta
  +Q(r\cos\theta,r\sin\theta)\sin\theta,\\
  \dot \theta&=\Theta(r,\theta):=\frac{1}{r}\Big(Q(r\cos\theta,r\sin\theta)\cos\theta
  -P(r\cos\theta,r\sin\theta)\sin\theta\Big).\nonumber
\end{align}

\begin{lemma}\label{le2} Let $\dot r=R(r,\theta), \, \dot
\theta=\Theta(r,\theta)$ be the expression~\eqref{eq3} of  system~\eqref{ee1} in
polar coordinates. Then the function $M_s$ given in~\eqref{ms} writes as

\begin{align}
  M_s&=  \frac{\partial V}{\partial x}P+\frac{\partial V}{\partial y}Q
  +s(\frac{\partial P}{\partial x}+\frac{\partial Q}{\partial y} )V\nonumber\\
&=\frac{\partial V}{\partial r}R+\frac{\partial V}{\partial \theta}\Theta
+s\left(\frac{\partial R}{\partial r} +\frac{\partial \Theta}{\partial
\theta}+\frac{R}{r}\right)V.\nonumber
\end{align}
\end{lemma}

\begin{theorem}\label{MTpolars}
Consider the planar differential system~\eqref{ee1},
\begin{equation*}\Dot x=P(x,y),\,\quad\Dot y= Q(x,y), \end{equation*}
where  $P$ and $Q$ are  real polynomials of degree $n$ and $P(0,0)=Q(0,0)=0.$
Define the polynomial
\begin{equation*}
p(r^2):=\frac1{2\pi r}\int_0^{2\pi} R(r,\theta)\,d\theta,
\end{equation*}
where $R$ is given in~\eqref{eq3} and set $w(r)=r^2p'(r^2)$. Denote by $d$ the
degree of $w$ and by  $N^+$ its number of  non-negative roots. For each fixed
$s\in\mathbb{R}$ consider the function
\begin{align*}\label{m}
M_s(r,\theta):&=R(r,\theta)w'(r)+s\left(\frac{\partial R(r,\theta)}{\partial r}
+\frac{\partial \Theta(r,\theta)}{\partial \theta}+\frac{R(r,\theta)}{r}\right)w(r)\nonumber\\
&=:\sum_{i=1}^{n+d-1} m_{i}(s,\theta)r^i,
\end{align*}
and, for any $i\ge1,$ let $\mu_i(s)$ be such that
$\max_{\theta\in[0,2\pi]}m_i(s,\theta)\le \mu_i(s).$

 Then, if the polynomial
\[\Phi_s(r):=\sum_{i=1}^{n+d-1} \mu_i(s)r^i\] is negative   for all
$r\in(0,\infty),$ system~\eqref{ee1} has at most $N^+$  limit cycles and  all of
them are hyperbolic.
\end{theorem}

\begin{proof}
We want to apply Corollary~\ref{B} to system \eqref{ee1}  with $V(x,y)=w(r)$ and
the value $s$ given in the statement of the Theorem. By hypothesis, we have
\[
 M_s=M_s(r,\theta)=\sum_{i=1}^{n+d-1} m_i(s,\theta)r^i\le
  \sum_{i=1}^{n+d-1} \mu_i(s)r^i=\Phi_{s}(r)<0
\]
for all $r\in(0,\infty).$ Notice that by the proof of Corollary~\ref{B} and because
$M_s$ does no vanish  there are no limit cycles in $\{w(r)=0\}.$ Hence the maximum
number of limit cycles is $\ell(\mathbb{R}^2)=0$ when $s\ge0$ and
$\ell(\mathbb{R}^2,w)$ if $s<0.$ In fact notice that $\{w(r)=0\}$ is formed by the
origin and $N^+-1$ disjoint concentric cercles. Therefore
$\ell(\mathbb{R}^2,w)=N^+$ and again by Corollary~\ref{B} the theorem follows.
\qed\end{proof}

\begin{remark}
(i) The choice of the function $V(x,y)=w(r)$ in Theorem~\ref{MTpolars} is motivated
by the following fact: for the  simple system that in polar coordinates writes as
$\dot r=rp(r^2), \,\dot\theta =q(r^2),$  where  $q$ is any arbitrary polynomial, it
holds that the corresponding $M_{-1}$, given in~\eqref{ms}, is always negative.

(ii) Following the proof of Corollary~\ref{B} it is not difficult to see that under
the hypotheses of the  Theorem~\ref{MTpolars}, if the system has only the origin as
a critical point then it has  at least $N^+-2$ limit cycles, with alternating
stability. The reason is that two consecutive circles of $\{w(r)=0\}$ always are
the boundaries of positive or negative invariant regions.
\end{remark}

We end this subsection with a concrete   application of Theorem~\ref{MTpolars} to
a 3-parameter family of planar vector fields. Consider the system
\begin{align}\label{eqex2}
  \dot x&=x(1-(x^2+y^2))(2-(x^2+y^2))-y+a x^2y +b x^2y^2,\nonumber\\
  \dot y&=x+y(1-(x^2+y^2))(2-(x^2+y^2))+cxy^2.
\end{align}
We will prove that if  $a,b$ and $c$ are such that
\[
\Psi_{a,b,c}(r):=-10+\frac94(|a|+|c|)+\frac94|b|r
+\left(12+|a|+|c|\right)r^2+|b|r^3-4r^4<0
\]
for all $r>0,$ then  system~\eqref{eqex2} has at most two (hyperbolic) limit
cycles. Moreover, when they exist, one is included in the disc
$\mathcal{D}:=\{x^2+y^2\le3/2\}$ and is stable and the other one is outside the
disc and it is  unstable.

To apply  Theorem \ref{MT} we compute $p(s)=2-3s+s^2.$ Then taking
 $w(r)=r^2p'(r^2)=r^2(-3+2r^2)$ and  $s=-1$, we obtain
\begin{align*}
M_{-1}(r,\theta)=&
\frac14\left(-40+a\left(6\sin(2\theta) -3\sin(4\theta)\right)+c\left(6\sin(2\theta) +3\sin(4\theta)\right)\right)r^4\\
&+\frac38b\left( 2\cos(\theta)-3\cos(3\theta)+\cos(5\theta)    \right)r^5\\
&+\left(12+a\sin(4\theta)-c\sin(4\theta)\right)r^6
-\frac{b}2\left(-\cos(3\theta)+\cos(5\theta)\right)r^7-4r^8.
\end{align*}
Hence, for the values of the parameters considered, we can prove that
 \[
M_{-1}(r,\theta)\le r^4\Psi_{a,b,c}(r)<0
\]
for all $r>0.$ Thus we can apply Theorem \ref{MT}.  Since $N^+=2$ we have proved
that  system~\eqref{eqex2}  has at most two (hyperbolic) limit cycles.

For instance the condition on $\Psi_{a,b,c}$ holds for $a=1/8$, $b=1/15$ and
$c=1/20$. Moreover for these parameters it is not difficult to prove, by using
resultants and the Sturm's theorem, that the origin is the unique critical point,
which is unstable. Finally, by studying the flow on $\{x^2+y^2=R^2\},$ for $R$ big
enough, and on $\{x^2+y^2=3/2\},$ we prove the existence of both limit cycles.

\section{More applications}
\label{sec:4}

This section contains an extension of a Massera's result extracted
from~\cite{GGiL} and a study of an extension of the van der Pol system introduced
in~\cite{Xian}.

\subsection{A generalization of a result of Massera}

Consider the generalized smooth second order Li\'{e}nard equation.
\begin{equation*}\label{a1}
\ddot x + f(x)\dot x+ g(x) = 0,
\end{equation*}
with $f$ and $g$ smooth functions. It can be written as the  planar system
\begin{equation}\label{a3}
\dot x = y, \qquad \dot y= -f(x)y-g(x).
\end{equation}
We define $G(x)=\int_0^x g(z)\,dz.$

Using once more the extended Bendixson-Dulac Theorem and its Corollary~\ref{B} we
can prove the following result.

    \begin{figure}[h]
\begin{center}
\begin{overpic}[scale=.53]{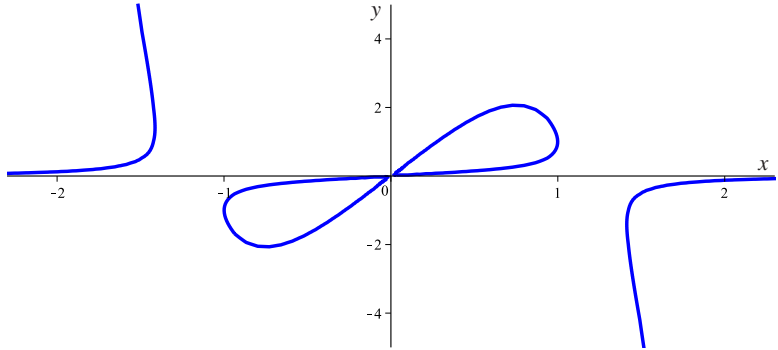}\put(97,23.5){$x$}\put(47.4,43.2){$y$}
\end{overpic}\vspace{-0.5cm}
\end{center}
\caption{Example of set $\{V=0\},$ under the hypotheses  of Proposition~\ref{cn2}.}
        \label{fig2}                %  opcional
    \end{figure}

\begin{proposition}\label{cn2} Let $\mathcal{W}=\mathcal{I}\times \mathbb{R}$ be
a vertical strip  of\, $\mathbb{R}^2,$ where $\mathcal{I}$ is an open interval
containing the origin. Assume that   the functions $f$ and $g$ are of class
$\mathcal{C}^1,$ that $g$ only vanishes at the origin and that  $f+
2(f/g)^{\prime}G$ does not change sign on ${\mathcal I}$,vanishing only at $x=0.$
Then system \eqref{a3} has at most one periodic orbit which entirely lies in
$\mathcal W$, and when it exists it is a hyperbolic limit cycle.
\end{proposition}

\begin{proof}
By taking $V(x,y)= y^2+ (2Gf)y/g+ 2G$ and $s=-1$ we can compute $M_s$ given
in~\eqref{ms}, obtaining that
\[M_{-1}= \left(f+ 2G(f/g)^{\prime} \right)y^2.\]
Notice that $G/g$ and $(f/g)^{\prime}G$ are well--defined at the origin. By the
hypotheses, $M_-1$ does not change sign on $\mathcal W$  and
$\{M_{-1}=0\}=\{xy=0\}.$ Moreover this set does not contain periodic orbits. Hence
we can apply Corollary~\ref{B}. Since $s<0$ we have to compute
$\ell(\mathcal{W},V).$ The function $V$ has degree 2 in $y$ and when $x=0$ the
only point in $\mathcal{V}_0:=\{V=0\}$ is $(0,0)$.  Therefore the set
 $\mathcal{V}_0$ has no oval
surrounding the origin. Moreover, since the origin is the only critical point of
the system and  $\mathcal{V}_0$ is without contact by the flow of the system, then
$\mathcal{V}_0$  does not contain  ovals at all. In Figure~\ref{fig2} we illustrate
a possible set $\mathcal{V}_0$, taking $f(x)=-4+x^2+x^4$ and $g(x)=x.$ Then
$V(x,y)=y^2+(-4+x^2+x^4)xy+x^2$ and $M_{-1}(x,y)=2(1+2x^2)x^2y^2.$ Hence, in
general,  all the connected regions of $\mathcal{W}\setminus\mathcal{V}_0$ are
simply connected but one and $\ell(\mathcal{W},V)=1$. Thus  we have proved the
uniqueness of the limit cycle. \qed
\end{proof}

We remark that Proposition~\ref{cn2} when $g(x)=x,$ contains the following
classical result, which was proved by Massera \cite{Ma} and Sansone \cite{Sa}.

\medskip

\noindent{\bf Massera's Theorem.} {\it Consider the Li\'{e}nard differential system
\eqref{a3} with $g(x)= x$, $f(0)<0$ and $f'(x)x>0$ if $x\neq 0$. Then system
\eqref{a3} has at most one limit cycle}.

\subsection{A generalization of van der Pol equation}

The system
\begin{equation}\label{sisb}
\dot{x}=y,\quad \dot{y}=-x+(b^2-x^2)(y+y^3),
\end{equation}
is introduced and studied in~\cite{Xian} as a generalization of the van der Pol
equation. In the papers~\cite{Han-Qian,Xian} it is proved that it has at most one
(hyperbolic) limit cycle and that it exists if and only if $b\in(0,b^*)$ for some
$0<b^*< \sqrt[6]{9\pi^2/16})\approx 1.33.$ This bifurcation value is refined in
\cite{GGG}, proving that $b^*\in(0.79,0.817)$.  In fact, numerically it can be seen
that $b^*\approx 0.80629$. In this section we will prove the uniqueness of the
limit cycle when $b\in(0,0.6]$ using a suitable Bendixson-Dulac function. This idea
is developed in~\cite{GGG} where the authors prove, with the same method, the
uniqueness and hyperbolicity of  the limit cycle holds when $b\in(0,0.817)$  and
its non-existence when $b\in[0.817,\infty).$

To give an idea of how we have found the function $V$ and the value $s$ to find the
function $M_s$ that we will use in our proof we  first study again the van der Pol
system. As we will see the main difficulty of this example is that  the function
$M_s$ is a function of two variables.

\subsubsection{The van der Pol equation (a second approach)}

The  van der Pol equation studied in Subsection~\ref{svdp}, after a rescaling of
variables, is equivalent to the system
\begin{equation}\label{vanderpol}
\dot{x}=y,\quad \dot{y}=-x+(b^2-x^2)y.
\end{equation}
Arguing like in Section~\ref{ss1} it is natural to start considering functions of
the form
\[
V(x,y)=v_2y^2+v_1(x)y+v_0(x),
\]
with $s=-1$. Then the corresponding $M_{-1}$ given in \eqref{ms} is a polynomial
of degree 2 in $y$, with coefficients being functions of $x$. In particular the
coefficient of $y^2$ is
\[
v_1'(x)+v_2(b^2-x^2).
\]
Taking $v_1(x)=(x^2-3b^2)v_2x/3$ we get that this coefficient vanishes. Next,
fixing $v_2=6$, and imposing to the coefficient of $y$ to be zero we obtain that
$v_0(x)=6x^2+c,$ for any constant $c.$ Finally, taking $c=b^2(3b^2-4),$ we arrive
to
\begin{equation}\label{vbvander}
V(x,y)=6y^2+2(x^2-3b^2)xy+6x^2+b^2(3b^2-4).
\end{equation}
Then
$$M_{-1}(x,y)=4x^4+b^2(3b^2-4)(x^2-b^2).$$

It is easy to see that for $b\in(0,2/\sqrt3)\approx(0,1.15)$, $M_{-1}(x,y)>0$.
Hence we can apply Corollary~\ref{B}. As $V(x,y)$ is quadratic in $y$, $V(x,y)=0$
has at most one oval, see Figure~\ref{f3} for $b=1.$  Hence
$\ell(\mathbb{R}^2,V)=1$ and we have proved the uniqueness and hyperbolicity of the
limit cycle when $b<2/\sqrt3.$ Recall that the proof given in Subsection~\ref{svdp}
is simpler and valid for all values of the parameter. We have included this one as
a motivation for the construction of the function $V(x,y)$ used to study
system~\eqref{sisb}.

    \begin{figure}[h]
\begin{center}
\begin{overpic}[scale=.45]{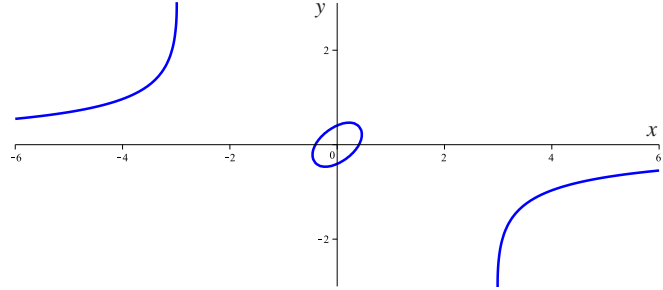}\put(97,23.5){$x$}\put(47,42){$y$}
\end{overpic}\vspace{-0.5cm}
\end{center}
\caption{The algebraic curve $V(x,y)=0$ with $b=1$.}
        \label{f3}                %  opcional
    \end{figure}

\subsection{System~\eqref{sisb} with $b\le 0.6$}\label{651}

By making some modifications to the function $V$ given by \eqref{vbvander}, we
propose the following function $V,$
\begin{align}\label{vv}
V(x,y)=&\big(2x^3+6b^2(1-b^2)x\big) y^3+6(1-b^2)y^2+2(x^2-3b^2)xy\nonumber\\&+
6(1-b^2)x^2+b^2(3b^2-4)
\end{align}
and again $s=-1$. Some computations give that
\begin{equation*}\label{Mb651}
\begin{array}{lll}
M_{-1}(x,y)&=&6\big((2-3b^2)x^4y^2-2b^2(2-b^2)x^3y^3+(2-b^2)x^2y^4\big)
+2(2-3b^2)x^4\\
&&-3b^2(14-15b^2)x^2y^2+12b^4(2-b^2)xy^3-b^2(4-9b^2)x^2\\
&&+3b^4(2-3b^2)y^2+b^4(4-3b^2).
\end{array}
\end{equation*}

Then we need to study the shape of the connected components of set
$\mathbb{R}^2\setminus\{V=0\}$ and the sign of  $M_{-1}(x,y)$. It can be seen that
the  algebraic curve $V(x,y) = 0,$ with $V(x,y)$ given by~\eqref{vv}, has no
singular points and at most one closed oval for $b\in(0.0.85]$. Moreover
$M_{-1}(x,y)$ does not vanish for $b\in(0,0.651)$. Hence we can apply again
Corollary~\ref{B}. Since $\ell(\mathbb{R}^2,V)=1$ we have proved the uniqueness and
hyperbolicity of the limit cycle for $b\le0.6$ (in fact for $b\le 0.651$).

The  tools used to prove  the above  assertions are given in~\cite{GGG}. Among
other methods the authors use discriminants, double discriminants, Sturm sequences
and the study of the points at infinity of the algebraic curves.

%The Bogdanov-Takens system
%\begin{equation}
%\dot x=y,\quad \dot y=-n+by+x^2+xy,
%\end{equation}
%was introduced in~\cite{Bog1975,Tak1974}. It  provides a universal unfolding of a
%cusp point of codimension 2; see also~\cite{GucHol83}. The uniqueness and
%hyperbolicy of its limit cycle was proved in~\cite{LiRouWan90} transforming it
%into a Li\'{e}nard system and then applying some results due to  Zhang Zhifen and
%Cherkas. In this section we give  a new proof using the Bendixson-Dulac approach.

\begin{acknowledgement}
The first author is supported by the MICIIN/FEDER grant number MTM2008-03437 and
the Generalitat de Catalunya grant number 2009SGR410
\end{acknowledgement}


\begin{thebibliography}{99.}%


\bibitem{a} M. J. \'{A}lvarez, A. Gasull, H.  Giacomini,
\emph{A new uniqueness criterion for the number of
 periodic orbits of Abel equations}, J. Differential Equations \textbf{234},
161--176  (2007).

\bibitem{a2} M. J. \'{A}lvarez, A. Gasull, R. Prohens,
\emph{Limit cycles for two families of cubic systems},
 Nonlinear Anal. \textbf{75},  6402--6417 (2012).

%\bibitem{Bog1975} R. I. Bogdanov,
%\emph {Versal deformation of a singular point of a vector field on the plane
%  in the case of zero eigenvalues  (Russian)},
%Funkcional Anal. i Prilozen \textbf{9}, 63 (1975).




\bibitem{cha} M. Chamberland, A. Cima, A. Gasull, F. Ma\~{n}osas. \emph{
Characterizing asymptotic stability with Dulac functions}, Discrete Contin. Dyn.
Syst. \textbf{17}, 59--76 (2007).

\bibitem{Che1} L. A. Cherkas, \emph{Estimation of the number
of limit cycles of autonomous systems}, Differential Equations \textbf{13},
529--547 (1977).

\bibitem{Che2} L. A. Cherkas, \emph{Dulac function for
polynomial autonomous systems on a plane}, Differential Equations \textbf{33},
692--701 (1997).

\bibitem{Che3} L. A. Cherkas, A. A. Grin', \emph{A second-degree polynomial
Dulac function for a cubic system on the plane}, Differential Equations
\textbf{33}, 1443--1445 (1997).

\bibitem{Che4} L. A. Cherkas, A. A. Grin', \emph{ A Dulac function in a half-plane
in the form of a polynomial of the second degree for a quadratic system},
Differential Equations \textbf{34}, 1346--1348 (1998).

\bibitem{Che5} L. A. Cherkas, A. A.  Grin', K. R.  Schneider,
\emph{Dulac-Cherkas functions for generalized Li\'{e}nard systems},
 Electron. J. Qual. Theory Differ. Equ. \textbf{35}, 23 pp. (2011).

\bibitem{Che6} L. A. Cherkas, A. A. Grin', \emph{On the Dulac function
for the Kukles system}, Differerential Equations \textbf{46},  818--826 (2010).

\bibitem{Che7} L. A. Cherkas, A. A. Grin', \emph{A function of limit cycles of
the second kind for autonomous functions on a cylinder},  Differential Equations
\textbf{47}, 462--470 (2011).


\bibitem{Chi}  C. Chicone, \emph{Ordinary differential equations with
applications}, Second edition, Texts in Applied Mathematics, 34. Springer, New
York, 2006.

\bibitem{CGM} A. Cima, A. Gasull, F. Ma\~{n}osas, \emph{Limit cycles for vector
fields with homogeneous components},  Appl. Math. (Warsaw) \textbf{24}, 281--287
(1997).


\bibitem{Conti} R. Conti, \emph{Soluzioni periodiche dell'equazione di Li\'enard
generalizatta. Esistenza ed unicit\`{a}}, Bolletino della Unione Matematica Italiana
\textbf{3}, 111--118 (1952).

\bibitem {DLA} { F. Dumortier, J. Llibre and J.C. Art\'{e}s},
{\it Qualitative theory of planar differential systems}, UniversiText,
Springer--Verlag, New York, 2006.

%
%\bibitem{DH} F. Dumortier, C.  Herssens, \emph{Polynomial Li\'{e}nard equations
%near infinity}, J. Differential Equations, \textbf{153}, 1--29 (1999).


\bibitem{F}  M. Fe{\v{c}}kan, \emph{A generalization of Bendixson's criterion},
Proc. Amer. Math. Soc. \textbf{129}, 3395--3399 (2001).


\bibitem{GGG} {J. D. Garc\'{\i}a-Salda\~{n}a, A. Gasull, H. Giacomini},
\emph{Bifurcation values for a family of planar vector fields of degree five},
preprint 2012.


\bibitem{GGi} { A. Gasull, H. Giacomini},
\emph{ A new criterion for controlling the number of limit cycles of some
generalized Li\'{e}nard equations}, J. Differential Equations \textbf{185},  54–-73
(2002).

\bibitem{GGi2} { A. Gasull, H. Giacomini},
\emph{Upper bounds for the number of limit cycles through linear differential
equations}, Pacific J. Math. \textbf{226},  277-–296 (2006).

\bibitem{GGi3} { A. Gasull, H. Giacomini},
\emph{Upper bounds for the number of limit cycles of some planar polynomial
differential systems}, Discrete Contin. Dyn. Syst. \textbf{27}, 217--229 (2010).



\bibitem{GGiL} { A. Gasull, H. Giacomini, J. Llibre},
\emph{New criteria for the existence and non-existence of limit cycles in Li\'{e}nard
differential systems}, Dyn. Syst. \textbf{24},  171--185 (2009).


\bibitem{GG} A. Gasull, A. Guillamon, \emph{Non-existence, uniqueness of
limit cycles and center problem in a system that includes predator-prey systems
and generalized Li\'enard equations}, Diff. Equations and Dynamical Systems
\textbf{3}, 345--366 (1995).


%\bibitem{Giacomini} H. Giacomini, S. Neukirch, \emph{Number of limit cycles
%of the Li\'enard equation}, Physical Review E, \textbf{56}, 3809--3813 (1997).

%\bibitem{GucHol83}
%J. Guckenheimer, P. Holmes, ``Nonlinear oscillations, dynamical systems, and
%bifurcations of
%  vector fields", \textbf{42} of Applied Mathematical Sciences.
%  Springer-Verlag, New York, 2002.

\bibitem{Han-Qian} M. Han, T. Qian, \emph{ Uniqueness of periodic
solutions for certain second-order equations}, Acta Math. Sin. (Engl. Ser.) {\bf
20},  247--254 (2004).

\bibitem{siam} S.B. Hsu, T.W. Huang, \emph{Global
stability for a class of predator-prey systems}  \textbf{55}, 763--783 (1995).


\bibitem{I}Yu.  Ilyashenko, \emph{Centennial history of
Hilbert's 16th problem}, Bull. Amer. Math. Soc. (N.S.) \textbf{39}, 301–-354
(2002).



\bibitem{K} Y. Kuang,
\textit{Global stability of Gause-type predator-prey systems}, J. Math. Biol.
\textbf{28}, 463--474 (1990).

\bibitem{LM} Y.  Li, J. S. Muldowney,
\textit{On Bendixson's criterion}, J. Differential Equations \textbf{106},
 27--39 (1993).


\bibitem{Lloyd} N. G. Lloyd, \emph{A note on the number of
limit cycles in certain  two-dimensional systems}, J. London Math. Soc. (2)
\textbf{20},   277--286 (1979).


%\bibitem{LiRouWan90}
%C. Li, C. Rousseau, X. Wang. \emph{A simple proof for the unicity
% of the limit cycle in the Bogdanov-Takens system},
%Canad. Math. Bull. \textbf{33}, 84--92 (1990).

\bibitem{Ma} { J.L. Massera},
{\emph Sur un th\'{e}or\`{e}me de G. Sansone sur l'\'{e}quation di Li\'{e}nard} (French), Boll.
Un. Mat. Ital. (3) {\bf 9}, 367--369 (1954).

\bibitem{MM} C. C.  McCluskey, J. S.  Muldowney, James S.
\emph{Bendixson-Dulac criteria for difference equations}, J. Dynam. Differential
Equations \textbf{10},  567--575 (1998).

\bibitem{Moreira} H.N. Moreira, \emph{On Li\'enard's
equation and the uniqueness of limit cycles in predator-prey systems}, J. Math.
Biol. \textbf{28},   341--354 (1990).




\bibitem{Per}L. M. Perko, ``Differential equations and dynamical systems",
 Third edition. Texts in Applied Mathematics, \textbf{7}. Springer-Verlag, New York, 2001.


\bibitem{Sa} { G. Sansone},
{\emph Soluzioni periodiche dell'equazione di Li\'{e}nard. Calcolo del periodo}
(Italian), Univ. e Politecnico Torino. Rend. Sem. Mat. {\bf 10}, 155--171 (1951).


%\bibitem{SB} J. Stoer, R. Bulirsch, ``Introduction to Numerical Analysis'',
%\emph{Springer Verlag NY}, 1980.

\bibitem{SC} G. Sansone, R.  Conti, ``Equazioni differenziali non lineari
(Italian)'', \emph{Edizioni Cremonese, Roma}, 1956.



\bibitem{Xian} X. Wang, J. Jiang, P. Yan, \emph{ Analysis of global
bifurcation for a class of systems of degree five}, J. Math. Anal. Appl. {\bf
222},  305--318 (1998).


%\bibitem{Tak1974} F. Takens, \emph{Singularities of vector fields}
%Inst. Hautes \'Etudes Sci. Publ. Math. \textbf{43}, 47--100 (1974).

\bibitem{Wilson} G. Wilson, \emph{Hilbert's sixteenth problem}, Topology
\textbf{17}, 53--73 (1978).

\bibitem{Yamato} K. Yamato, \emph{An effective method of counting the
number of limit cycles}, Nagoya Math. J. \textbf{76}, 35--114  (1979).

\bibitem{Ye} Yan Qian Ye \& others, ``Theory of limit cycles'', Translations of
Mathematical Monographs \textbf{66}. American Mathematical Society, Providence,
RI, 1986.



\bibitem{Z2}
 Zhi Fen Zhang \& others, ``Qualitative theory of
differential equations'', Translations of Mathematical Monographs \textbf{101}.
American Mathematical Society, Providence, RI, 1992.








%
\end{thebibliography}
\end{document}